\documentclass[11pt,draft]{amsart}
\usepackage{amssymb, amstext, amscd, amsmath}
\usepackage[all]{xy}
\makeatletter
\def\@cite#1#2{{\m@th\upshape\bfseries%
[{#1\if@tempswa{\m@th\upshape\mdseries, #2}\fi}]}}
\makeatother
\theoremstyle{plain}
\newtheorem{theorem}{Theorem}[section]
\newtheorem{corollary}[theorem]{Corollary}
\newtheorem{proposition}[theorem]{Proposition}
\newtheorem{lemma}[theorem]{Lemma}

\newtheorem{problem}{Problem}
\theoremstyle{definition}
\newtheorem{definition}[theorem]{Definition}
\newtheorem{example}[theorem]{Example}
\newtheorem{remark}[theorem]{Remark}
\theoremstyle{remark}
\newcommand{\bbC}{{\mathbb{C}}}

\newcommand{\bbI}{{\mathbb{I}}}
\newcommand{\bbF}{{\mathbb{F}}}

\newcommand{\bbN}{{\mathbb{N}}}
\newcommand{\bbQ}{{\mathbb{Q}}}

\newcommand{\bbT}{{\mathbb{T}}}
\newcommand{\bbZ}{{\mathbb{Z}}}

\newcommand{\A}{{\mathcal{A}}}
\newcommand{\B}{{\mathcal{B}}}
\newcommand{\C}{{\mathcal{C}}}

\newcommand{\F}{{\mathcal{F}}}
\newcommand{\G}{{\mathcal{G}}}
\renewcommand{\H}{{\mathcal{H}}}

\newcommand{\K}{{\mathcal{K}}}
\renewcommand{\L}{{\mathcal{L}}}

\renewcommand{\O}{{\mathcal{O}}}

\newcommand{\T}{{\mathcal{T}}}

\newcommand{\X}{{\mathcal{X}}}

\newcommand{\Z}{{\mathcal{Z}}}

\newcommand{\fA}{{\mathfrak{A}}}
\newcommand{\fB}{{\mathfrak{B}}}
\newcommand{\fC}{{\mathfrak{C}}}

\newcommand{\fG}{{\mathfrak{G}}}

\renewcommand{\phi}{\varphi}
\newcommand{\upchi}{{\raise.35ex\hbox{\ensuremath{\chi}}}}

\newcommand{\alg}{\operatorname{alg}}

\newcommand{\End}{\operatorname{End}}

\newcommand{\id}{{\operatorname{id}}}

\newcommand{\lat}{\operatorname{lat}}

\newcommand\supp{\mathop{\rm supp}}

\newcommand\Mat{{\bf \mathop{\rm M}}}

\newcommand{\ca}{\mathrm{C}^*}

\newcommand{\Gvertex}{{\mathcal{G}}^{0}}
\newcommand{\Gxvertex}{{\mathcal{G}}^{0}_{v}}
\newcommand{\Gedge}{{\mathcal{G}}^{1}}
\newcommand{\Gxedge}{{\mathcal{G}}^{1}_{v}}
\newcommand{\Ginfty}{{\mathcal{G}}^{\infty}}

\newcommand{\sot}{\textsc{sot}}

\newcommand{\sca}[1]{\left\langle#1\right\rangle}
\begin{document}

\title[Local maps and representations]{Local maps and the representation theory of operator algebras}

\author[E.G. Katsoulis]{Elias~G.~Katsoulis}
\address {Department of Mathematics
\\East Carolina University\\ Greenville, NC 27858\\USA}
\email{katsoulise@ecu.edu}

\thanks{2010 {\it  Mathematics Subject Classification.}
46L08, 47B49, 47L40, 47L65}
\thanks{{\it Key words and phrases:} local derivation, local multiplier, reflexivity, topological graph, tensor algebra, $\ca$-correspondence, Cuntz-Pimsner $\ca$-algebra.}

\maketitle

%%%%%%%%%%%%%%%%
\begin{abstract}
Using representation theory techniques we prove that various spaces of derivations or one-sided multipliers over certain operator algebras are reflexive. A sample result: any bounded local derivation (local left multiplier) on an automorphic semicrossed product $C(\Omega) \times_{\sigma} \bbZ^{+}$ is a derivation (resp. left multiplier). In the process we obtain various results of independent interest. In particular we show that the finite dimensional nest representations of the tensor algebra of a topological graph separate points.
\end{abstract}
%%%%%%%%%%%%%%%%
\section{Introduction}
%%%%%%%%%%%%%%%%

One of the early results of Barry Johnson \cite{Joh} asserts that if $\A$ is a semi simple Banach algebra\footnote{In this paper all Banach algebras have an approximate identity and all linear mappings are necessarily bounded.} and $S:\A \rightarrow \A$ is an operator that leaves invariant all closed left ideals of $\A$, then $S$ is a left multiplier of $\A$. This result has been the source of inspiration for subsequent work, including its recent use by the author \cite{Kat} to cast new light on a familiar behavior for the adjointable operators on a Hilbert $\ca$-module.

It is easy to see that the maps preserving all closed left ideals of a Banach algebra $\A$ coincide with the approximately local left multipliers on $\A$. If $\X$ is a right $\A$-module then a map $S: \A \rightarrow \X$ is said to be an \textit{approximately local  left multiplier} iff for any $A \in \A$ there exists a sequence $\{ X_{A,n} \}_n$ in $\X$ so that $S(A) = \lim_{n} X_{A, n} A$. One of the motivating questions for the present work is to what extend Johnson's Theorem is valid beyond semisimple operator algebras, i.e., for which operator algebras all approximately local left multipliers are left multipliers. Equivalently, we ask for which operator algebras $\A$, the algebra $LM(\A)$ consisting of all left multipliers on $\A$ is reflexive, i.e., $\alg \lat LM(\A) = LM(\A)$. (Here we view $LM(\A)$  as a subalgebra of all bounded operators on $\A$.)

This  line of investigation is not new. Don Hadwin, Jiankui Li and their collaborators \cite{Had, HadK, HadL2, HadL, LP} have been investigating questions of this type for the past 20 years for various reflexive operator algebras through the use of idempotents and their separating spaces.  In particular Hadwin and Li \cite{HadL} have shown that Johnson's Theorem holds for all CSL algebras. Also Bresar, Semrl and others have been investigating local multipliers (and derivations) in various other settings, including the purely algebraic, again through an intricate use of idempotents \cite{Bresar, BS1}. The present paper advocates the more systematic use of representation theory on this line of investigation and adds new examples which were previously inaccessible to the list of operator algebras for which Johnson's Theorem is valid. In particular, Theorem \ref{topolgraphmult} shows that all approximately local  left multipliers for the tensor algebra of a topological graph are actually left multipliers. This applies to various algebras that are currently under investigation, including Peters' semicrossed products, the tensor algebras of multivariable dynamics, subalgebras of Stacey's and Exel's crossed products and the tensor algebras of graphs. Theorem \ref{topolgraphmult} is proven with the use of representation theory and a crucial step in its proof is a result of independent interest: the finite dimensional nest representations of the tensor algebra of a topological graph separate points (Theorem \ref{nestrepn}). This generalizes one of the main results of an earlier work of Davidson and the author \cite[Theorem 4.7]{DavKatProcLMS}.

It turns out that the problem of deciding whether approximately local multipliers are actually multipliers is intimately related to the study of local derivations. The concept of a local derivation was introduced by Dick Kadison in his seminal paper \cite{Kad} and was further studied by Johnson \cite{Joh2}, Larson \cite{Larson}, Larson and Sourour \cite{LarsonSour} and many others. If $\X$ is a bimodule of a Banach algebra $\A$ then a map $\delta: \A \rightarrow \X$ is said to be a \textit{local derivation} iff for every $A \in \A$, there exists derivation $\delta_A$ so that $\delta(A) = \delta_A (A)$. (The obvious approximate version of this definition establishes the concept of approximately local derivation.) The question of whether all approximately local derivations of $\A$ are actually derivations is equivalent to the reflexivity (in the sense of \cite{Larson}) of the space $\Z_1(\A)$ of all derivations and includes as a special case the corresponding one for local derivations. Johnson established the much stronger result that the space $\Z_1(\A, \X)$ of all derivations from a $\ca$-algebra $\A$ into an $\A$-module is reflexive thus generalizing an earlier result of Kadison \cite{Kad}. Hadwin and Li \cite{HadL2} and Crist \cite{Crist} have established the reflexivity for $\Z_1(\A)$ in the case where $\A$ is a CD CSL algebra or a direct limit of finite dimensional CSL algebras respectively. Beyond these two classes, very little is known regarding the reflexivity of $\Z_1(\A)$ for a non-selfadjoint operator algebra $\A$. (Note however \cite{LP}.)

The other central result of this paper, Theorem \ref{mainder}, establishes that approximately local derivations are derivations for a large class of tensor algebras of topological graphs. The proof is more involved than that of Theorem \ref{topolgraphmult} and introduces new tools, including that of an acyclic discrete graph and its associated representations. It turns out that if there are sufficiently many acyclic or transitive discrete graphs associated with a topological graph, then one can build a nice representation theory for its tensor algebra (see Lemmas \ref{idemotrep} and \ref{kernelsquare}). Our theory applies in particular to Peters' semicrossed products $C_0(\Omega ) \times_{\sigma} \bbZ^+$ when $\sigma$ is a homeomorphism on $\Omega$ \cite{Arv, DavKatCrelle, Pet}, to the tensor subalgebras of Exel and Vershik's crossed products by topologically free coverings \cite{BR, BRV, EV}, to tensor algebras of Exel's crossed products by partial automorphisms \cite{Exel} and to various tensor algebras of multivariable dynamical systems \cite{DavKatMem}. In all these cases, approximately local derivations are derivations. We also obtain a similar result for various non-selfadjoint algebras which do not come from $\ca$-correspondences. In particular, Proposition \ref{RFDder} implies that any local derivation on the universal non-selfadjoint algebra generated by $n$ contractions is a derivation.

 Unlike the case of $\ca$-algebras \cite{Joh2} or CSL algebras \cite{HadL}, a local derivation on a tensor algebra $\A$ with range in a Banach $\A$-bimodule need not be a derivation. Counterexamples to demonstrate that behavior are presented in the last section of the paper. The author welcomes this phenomenon and believes that it will lead to computable isomorphism invariants for these algebras. This will be explored in a subsequent work.

%%%%%%%%%%%%%%%%%%%%%%%%%%%%%%%%%%
\section{Revisiting the General Theory}
%%%%%%%%%%%%%%%%%%%%%%%%%%%%%%%%%%%%%%

In this section we gather various results needed for the rest of the paper. Theorem \ref{basic1} and \ref{basic2} are due to Don Hadwin, Jiankui Li and their collaborators \cite{Had, HadK, HadL2, HadL, LP}. Both results have appeared repeatedly in the literature in various forms and degrees of generality. It turns out that in the form needed here, both have quite elementary proofs (even in the non-unital case) which we include. Examples \ref{nonmult} and \ref{nonder} at the end of the section appear to be new; these are quite crucial for building the counterexamples of the subsequent sections.

\begin{theorem} \label{basic1}
Let $\A$ be a Banach algebra generated \textup{(}as a Banach space\textup{)} by its idempotents and $\X$ be a right Banach $\A$-module. Then any approximately left multiplier from $\A$ into $\X$ is a multiplier. Hence $LM(\A, \X)$ is reflexive.
\end{theorem}

\begin{proof}
Let $S: \A \rightarrow \X$ be an approximate left multiplier. Note that for any $A ,B, P \in \A$ with $P=P^2$,  we have $S(ABP) \in \overline{\X BP}$ and $S(AB(I-P)) \in \overline{\X (B-BP)}$ Therefore
\begin{align*}
S(AB)P&= S(ABP)P+S(AB(I-P))P \\
	 	  &=S(ABP)P= S(ABP).
\end{align*}
Letting $B$ range over an approximate unit for $\A$, we obtain $S(AP)=S(A)P$. Since $\A$ is generated by its idempotents, $S$ is a left multiplier, as desired.
\end{proof}

One can easily see that a slight modification of the above proof also implies that the space $\End_{\A}(\X)$ of right $\A$-module operators on $\X$ is also reflexive. We leave the details to the reader.

\begin{lemma}
Let $\X$ be a Banach $\A$-bimodule for a unital Banach algebra $\A$ and let $\delta: \A \rightarrow \X$ be an approximately local derivation. Then
\[
\delta(PQ)= \delta(P)Q+P\delta(Q)
\]
for any idempotents $P, Q \in \A$.
\end{lemma}

\begin{proof}
Since $\delta$ is an approximately local derivation,  \[
P^{\perp}\delta(PQ)Q^{\perp}=P\delta(P^{\perp}Q)Q^{\perp}=0,
\] where $P^{\perp}\equiv I-P$ and similarly for $Q^{\perp}$. Hence
\begin{equation} \label{der1}
\delta(PQ)Q^{\perp}=P\delta(PQ)Q^{\perp}
                =P\delta(Q)Q^{\perp}
\end{equation}
and similarly
\[
\delta(PQ^{\perp})Q= P\delta(Q^{\perp})Q.
\]
By replacing in the above equation $Q^{\perp}$ with $I-Q$ and then distributing, we obtain
\[
\delta(P)Q -\delta(PQ)=P\delta(I)Q-P\delta(Q)Q
\]
and since $\delta(I)=0$,
\begin{equation} \label{der2}
\delta(PQ)Q =\delta(P)Q + P\delta(Q)Q.
\end{equation}
The conclusion follows by adding (\ref{der1}) and (\ref{der2}).
\end{proof}

\begin{theorem} \label{basic2}
Let $\A$ be a Banach algebra generated by its idempotents and $\X$ an Banach $\A$-bimodule. Then $\Z_1(\A, \X)$ is reflexive.
\end{theorem}

\begin{proof}
Extend $\delta$ by linearity to an approximate derivation $\hat{\delta}$ of the unitazation $\A_1$ of $\A$, by setting $\hat{\delta}(I)=0$. Then the extension $\hat{\delta}$, and therefore $\delta$ itself, satisfies the conclusion of the above result for $P , Q \in \A$. Since $\A$ is generated by its idempotents, the result follows.
\end{proof}

In the papers \cite{Had, HadK, HadL2, HadL, LP, Samei} the reader will find more general results regarding local maps on Banach algebras. We now present some counterexamples which demonstrate that even in the simplest of situations, local maps may fail the corresponding global property.

\begin{example} \label{nonmult}
If
\[
\fA= \left\{
\begin{pmatrix}
\lambda & \mu \\
0 &\lambda
\end{pmatrix} \mid \lambda, \mu \in \bbC
\right\} = \left\{ \lambda I + \mu E_{12} \mid \lambda, \mu \in \bbC
\right\}
\]
then
\[
S_{\fA}:\fA \longrightarrow \fA:
\lambda I + \mu E_{12} \longmapsto
\lambda I + 2 \mu E_{12}
\]
is a local multiplier which is not a multiplier.

Indeed if $\lambda \neq 0$ then $S_{\fA}(\lambda I + \mu E_{12}) = (I + \mu / \lambda E_{12})(\lambda I + \mu E_{12})$ or otherwise $S_{\fA}(\mu E_{12}) = 2I (\mu E_{12}) $. This shows that $S_{\fA}$ is a local multiplier. It is easy to see that in the case $\lambda \neq 0$, the factor $(I + \mu / \lambda E_{12})$ is uniquely determined by $\lambda I + \mu E_{12}$ and so $S_{\fA}$ cannot be a multiplier.
\end{example}

\begin{example} \label{nonder}
Let $\{ E_{ij}\}_{i, j =1}^{3}$ be the standard matrix unit system of $\Mat_3 (\bbC)$ and let $N = E_{12} +E_{23}$.

If
\[
\fB =
\left\{
\begin{pmatrix}
\lambda & \mu &\nu\\
0 &\lambda &\mu \\
0 & 0 &\lambda
\end{pmatrix} \mid \lambda, \mu , \nu \in \bbC
\right\} = \left\{ \lambda I + \mu N + \nu N^2 \mid \lambda, \mu , \nu\in \bbC
\right\}
\]
then
\[
\delta_{\fB}:\fB \longrightarrow \fB:
\lambda I + \mu N + \nu N^2  \longmapsto \nu N^2
\]
is a local derivation which is not a derivation.

Indeed, it is easy to see that the mappings
\[
\delta_1:\fB \longrightarrow \fB:
\lambda I + \mu N + \nu N^2  \longmapsto \mu N^2
\]
and
\[
\delta_2:\fB \longrightarrow \fB:
X \longmapsto (E_{11} +E_{22})X - X(E_{11} +E_{22})
\]
are derivations. Furthermore, if $X= \lambda I + \mu N + \nu N^2$, then \[
\delta_{\fB}(X) = \left\{ \begin{array}{cl}
(\nu /\mu )\, \delta_1(X) &\mbox{ if $\mu\neq 0$} \\
\delta_2(X) &\mbox{ otherwise}.
\end{array} \right.
\]
Therefore $\delta_{\fB}$ is a local derivation.

However, $\delta_{\fB}$ is not a derivation. Indeed, if $X= \lambda I + \mu N + \nu N^2$ and $X'= \lambda' I + \mu' N + \nu' N^2$, then the $(1,3)$-entry of \[\delta_{\fB}(XX')-\delta_{\fB}(X)X'-X\delta_{\fB}(X')\] is equal to $\mu\mu' $, which is not equal to $0$ in general.
\end{example}

\begin{example}[Crist \cite{Crist}] \label{Crist}
Let $\{E_{ij} \}$ be the
standard matrix unit system for $\Mat_3 (\bbC)$. If
\[
\fC = \bbC I+[E_{12} , E_{13} , E_{23}].
\]
and
\[
\delta_{\fC}:\fC \longrightarrow \fC:
\sum x_{ij} E_{ij} \longmapsto (2x_{13}-x_{12} +x_{23})E_{13}
\]
then $\delta_{\fC}$ is a local derivation which is not a derivation.
\end{example}

%%%%%%%%%%%%%%%%%%%%%%%%%%%%%
\section{Tensor algebras, nest representations and local multipliers}
%%%%%%%%%%%%%%%%%%%%%%%%%%%%%%%%%%%%%%%%

Davidson and the author have shown in \cite[Theorem 4.7]{DavKatProcLMS} that the finite dimensional nest representations of the tensor algebra of a countable graph separate points. In this section we extend this result to all tensor algebras of topological graphs. This applies in particular to algebras associated to one or more dynamical systems such as Peters' semicrossed products.  Combined with the results of the previous section, this result allows us to obtain definitive information regarding the local multipliers on these algebras.

A $\ca$-correspondence $(X,\A,\phi_X)$ consists of a $\ca$-algebra $\A$, a Hilbert $\A$-module $(X, \sca{\phantom{,},\phantom{,}})$ and a
(perhaps degenerate) $*$-homomorphism $\phi_X\colon \A \rightarrow \L(X)$.

A (Toeplitz) representation $(\pi,t)$ of a $\ca$-correspondence into
a $\ca$-algebra $\B$, is a pair of a $*$-homomorphism $\pi\colon \A \rightarrow \B$ and a linear map $t\colon X \rightarrow \B$, such that
\begin{enumerate}
 \item $\pi(a)t(\xi)=t(\phi_X(a)(\xi))$,
 \item $t(\xi)^*t(\eta)=\pi(\sca{\xi,\eta})$,
\end{enumerate}
for $a\in \A$ and $\xi,\eta\in X$. A representation $(\pi , t)$ is said to be \textit{injective} iff $\pi$ is injective; in that case $t$ is an isometry.

The $\ca$-algebra generated by a representation $(\pi,t)$ equals the closed linear span of $t^n(\bar{\xi})t^m(\bar{\eta})^*$, where for simplicity $\bar{\xi}\equiv (\xi^{1},\dots,\xi^{(n)})\in X^n$ and $t^n(\bar{\xi})\equiv t(\xi_1)\dots t(\xi_n)$.
For any
representation $(\pi,t)$ there exists a $*$-homomorphism
$\psi_t:\K(X)\rightarrow B$, such that $\psi_t(\Theta^X_{\xi,\eta})=
t(\xi)t(\eta)^*$.

It is easy to see that for a $\ca$-correspondence $(X,\A,\phi_X)$ there exists universal Toeplitz representation, denotes as $(\pi_{\infty} , t_{\infty})$, so that any other representation of $(X,\A,\phi_X)$ is equivalent to a direct sum of subrepresentations of $(\pi_{\infty} , t_{\infty})$. We define the Cuntz-Pimsner-Toeplitz $\ca$-algebra $\T_X$ as the $\ca$-algebra generated by all elements of the form $\pi_{\infty}(a), t_{\infty}(\xi)$, $a \in \A$, $\xi \in X$. The algebra $\T_X$ satisfies the following universal property: for any Toeplitz representation $(\pi,t)$ of $X$, there exists a representation $\pi \times t$ of $\T_X$ so that $\pi(a) = \big((\pi \times t)\circ \pi_{\infty}\big)(a)$, $\forall a \in \A$, and $t(\xi) = \big((\pi \times t)\circ t_{\infty}\big)(\xi)$, $\forall \xi \in X$.

\begin{definition}
The \emph{tensor algebra} $\T_{X}^+$ of a $\ca$-correspondence \break
$(X,A,\phi_X)$ is the norm-closed subalgebra of $\T_X$ generated by
all elements of the form $\pi_{\infty}(a), t_{\infty}(\xi)$, $a \in A$, $\xi \in X$.
\end{definition}

It is worth mentioning here that $\T^+_{\X}$ also sits naturally inside the Cuntz-Pimsner algebra $\O_{\X}$ associated with the $\ca$-correspondence $X$; this folloes from \cite{KatsoulisKribsJFA, MS}. As we will not be making essential use of that theory here, we skip the pertinent definitions and results and instead direct the reader to  \cite{KatsoulisKribsJFA, MS} for more details.

The tensor algebras for $\ca$-correspondences were pioneered by Muhly and Solel in \cite{MS}. They form a broad class of non-selfadjoint operator algebras which includes as special cases Peters' semicrossed products \cite{Pet}, Popescu's non-commutative disc algebras \cite{Pop}, the tensor algebras of graphs (introduced in \cite{MS} and further studied in \cite{KaK} and the tensor algebras for multivariable dynamics \cite{DavKatMem}, to mention but a few.

Due to its universality, the Cuntz-Pimsner-Toeplitz $\ca$-algebra $\T_X$ admits a gauge action that leaves $\pi_{\infty}(A)$ elementwise invariant and ``twists" each  $t_{\infty}(\xi)$, $\xi \in X$, by a unimodular scalar. Using this action, and reiterating a familiar trick with the Fejer kernel, one can verify that each element $T \in \T_X^+$ admits a Fourier series expansion
\begin{equation} \label{Cesaro}
T = \pi_{\infty}(a) +\sum_{i=1}^{\infty} \,  t_{\infty}^n(\bar{\xi}_n), \quad a \in A, \, \bar{\xi}_n \in \X^n, n=1,2, \dots,
\end{equation}
where the summability is in the Cesaro sense. (See \cite{Katsura, MS} for more details.)

A broad class of $\ca$-correspondences arises naturally from the concept of a topological graph. A topological graph $\G= (\G^{0}, \G^{1}, r , s)$ consists of two $\sigma$-locally compact\footnote{Due to this assumption, all discrete graphs appearing in this paper are countable.} spaces $\Gvertex$, $\Gedge$, a continuous proper map $r: \Gedge \rightarrow \Gvertex$ and a local homeomorphism $s: \Gedge \rightarrow \Gvertex$. The set $\Gvertex$ is called the base (vertex) space and $\Gedge$ the edge space. When $\Gvertex$ and $\Gedge$ are both equipped with the discrete topology, we have a discrete countable graph (see below).

With a topological graph $\G=  (\G^{0}, \G^{1}, r , s)$ there is a $\ca$-correspondence $X_{\G}$ over $C_0 (\Gvertex)$. The right and left actions of $C_0(\Gvertex )$ on $C_c ( \Gedge)$ are given by
\[
(fFg)(e)= f(r(e))F(e)g(s(e))
\]
for $F\in C_c (\Gedge)$, $f, g \in C_0 (\Gvertex)$ and $e \in \Gedge$. The inner product is defined for $F, G \in C_c ( \Gedge)$ by
\[
\left< F \, | \, G\right>(v)= \sum_{e \in s^{-1} (v)} \overline{F(e)}G(e)
\]
for $v \in \Gvertex$. Finally, $X_{\G}$ denotes the completion of $C_c ( \Gedge)$ with respect to the norm
\begin{equation} \label{norm}
\|F\| = \sup_{v \in \Gvertex} \left< F \, | \, F\right>(v) ^{1/2}.
\end{equation}

When $\Gvertex$ and $\Gedge$ are both equipped with the discrete topology, then the tensor algebra $\T_{\G}^+ \equiv \T^+_{X_{\G}}$ associated with $\G$ coincides with the quiver algebra of Muhly and Solel \cite{MS}. In that case, $\T_{\G}^+$ has a natural presentation which we now describe.

Let $\G=  (\G^{0}, \G^{1}, r , s)$ be a countable discrete graph and let $\Ginfty$ be the path space of $\G$. This consists of all vertices $v \in \Gvertex$ and all paths $p = e_ke_{k-1}\dots e_1$, where the $e_i$  are edges satisfying $s(e_i)=r(e_{i-1})$, $i=1,2, \dots , k$, $k \in \bbN$. (Paths of the form $p = e_ke_{k-1}\dots e_1$ are said to have length $|p|=k$ and vertices are called paths of length $0$.) The maps $r$ and $s$ extend to $\Ginfty$ in the obvious way, two paths $p_1$ and $p_2$ are composable whenever $s(p_2)=r(p_1)$ and in that case, the composition $p_2p_1$ is just the concatenation of $p_1$ and $p_2$. Let $\{\xi_p\}_{p \in \Ginfty}$ denote the usual orthonormal basis of the Fock space $\H_{\G} \equiv l^2( \Ginfty )$, where $\xi_p$ is the characteristic function of $\{p\}$. The left creation operator $L_q$, $q \in \Ginfty$, is defined by
\[
L_q \xi_p =
 \left\{ \begin{array}{cl}
\xi_{qp} &\mbox{ if $s(q) = r(p)$} \\
0 &\mbox{ otherwise}.
\end{array} \right.
\]
By \cite[Corollary 2.2]{FowR}, the algebra generated by $\{ L_p \mid p \in \Ginfty\}$ is (completely isometrically) isomorphic to $\T_{\G}^+$. Its weak closure, denoted as $\L_{\G}$, is the familiar free semigroupoid algebra of $\G$, first studied by Kribs and Power~\cite{KrP}.

Another class of examples for topological graphs arises from multivariable dynamical systems. If $\Omega$ is a $\sigma$-locally compact Hausdorff space equipped with $n$ proper continuous self maps $\sigma = (\sigma_1, \sigma_2 , \dots, \sigma_n )$ then let $\Gvertex = \Omega$, $\Gedge = \{1, 2, \dots , n\} \times \Omega$, let $s: \Gedge \rightarrow \Gvertex$ be the natural projection and let $r(i, x)=\sigma_i(x)$. This defines a topological graph denoted as $\G(\Omega , \sigma)$ and such graphs are naturally associated to various familiar operator algebras. Indeed in the case of a single variable dynamical system $(\Omega, \sigma)$, the tensor algebra $\T_{\G(\Omega , \sigma)}^+$ is completely isometrically isomorphic to Peter's semicrossed product $C_0(\Omega) \times_{\sigma} \bbZ^+$ \cite{Pet}, which was first studied by Arveson \cite{Arv}. This is the universal non-selfadjoint operator algebra generated by a copy of $C_0(\Omega) $ and an isometry $V$, subject to the covariance relations $fV= V (f\circ \sigma)$ , $ f \in C_0(\Omega)$. For more general dynamical systems $(\Omega, \sigma)$, the tensor algebra $\T_{\G(\Omega , \sigma)}^+$, usually denoted as $\T^+(\Omega, \sigma)$, was first studied by Davidson and the author in \cite{DavKatMem}.

The following generalizes one of the main results of an earlier work of Davidson and the author \cite[Theorem 4.7]{DavKatProcLMS} and it is essential for the main result of this section, Theorem \ref{topolgraphmult}.

 \begin{theorem} \label{nestrepn}
If $\G=  (\G^{0}, \G^{1}, r , s)$ is a topological graph, then the finite dimensional nest representations of its tensor algebra $\T_{\G}^{+}$ separate points.
\end{theorem}

\begin{proof}
We will construct a family of representations $(\pi_v , t_v)$, $v \in \Gvertex$, of the correspondence $X_{\G}$ so that the integrated representations $\phi_v\equiv \pi_v \times t_v$ satisfy:
\begin{itemize}
\item[(i)] for each $v \in \Gvertex$ there exists a discrete graph $\G_v$ so that \begin{equation} \label{overl}\overline{\phi_v(\T_{\G}^+)}^{\textsc{ wot}} = \L_{\G_v},\end{equation}
    where $\L_{\G_v}$ is the free semigroupoid algebra associated with $\G_v$
\item[(ii)] the representations $\phi_v$, $v \in \Gvertex$, separate the points in $\T_{\G}^{+}$.
    \end{itemize}
Once this is done, the proof follows from the result of Davidson and Katsoulis \cite[Theorem 4.7]{DavKatProcLMS} that the weakly continuous finite dimensional nest representations of the free semigroupoid algebras separate points.

Fix a $v \in \Gvertex$ and define inductively
\begin{align*}
\G_{v , \, n+1}^{0}&=r\left(\G_{v , \, n}^{1}\right) \\
\G_{v , \, n+1}^{1}&=s^{-1}\left(\G_{v , \, n+1}^{0}\right), \quad n=0, 1, \dots,
\end{align*}

where $\G_{v , 0}^{0}=\{ v \}$ and $\G_{v ,  0}^{1}=s^{-1}(\{ v \})$. If
\[
\G_{v}^{i} = \bigcup_{n\geq0} \G_{v ,  n}^i, \quad i=0,1,
\]
then we define $\G_v \equiv (\Gxvertex , \Gxedge, r , s)$, where $r$ and $s$ are the restrictions on $\Gxedge$ of the corresponding maps coming from $\G= (\G^{0}, \G^{1}, r , s)$. (The family of discrete graphs $\G_v$, $ v \in \G^0$, appearing in this proof is said to be the \textit{family of discrete graphs associated with} $\G$.)

Let $\G_{v}^{\infty}$ be the (finite) path space of $\G_v$ and let
$\H_{\G_v}$ be the Fock Hilbert space associated with $\G_v$. For $f \in C_0 (\Gvertex)$ and $F \in C_c(\Gedge)$ we define
 \begin{align*}
 \pi_v (f)\xi_p &= f(r(p))\xi_p \\
 \mbox{and} \phantom{XXXXXXXXXXXXX}&\phantom{XXXXXXXXXXXXXXXXXXXXX} \\
t_v(F)\xi_p &= \sum_{e \in s^{-1}(r(p) )}\, F(e)\xi_{ep}, \quad p \in \G_{v}^{\infty}
\end{align*}
with the understanding that $r(p) =p$, when $p \in \Gxvertex$.

Clearly,
\[
\| t_v(F)\xi_p \,\|^2 = \| \left< F \, | \, F\right>(r(p)) \|\leq \|F \|
\]
for any $p \in \G_{v}^{\infty}$. Furthermore, if $p, p' \in \G_{v}^{\infty}$ are distinct paths then the vectors $t_v(F)\xi_p$ and $t_v(F)\xi_{p'}$ are orthogonal. Hence $t_v(F)$ extends to a bounded operator on $\H_{\G_v}$. Its adjoint satisfies $t_v^*(F)\xi_w = 0$, when $w \in \Gxvertex$, and
\[
t_v^*(F)\xi_{ep}= \overline{F(e)}\xi_{p}
\]
otherwise. Therefore,
\begin{align*}
t_v(F)^* t_v (G) \xi_p&= t_v(F)^* \Big( \sum_{e \in s^{-1}(r(p) )}\, G(e)\xi_{ep}\Big) \\
&=\sum_{e \in s^{-1}(r(p) )}\,\overline{F(e)} G(e)\xi_{p} \\
&= \left< F \, | \, G \right> (r(p))\xi_p = \pi_v(\left< F \, | \, G \right> ) \xi_p. \end{align*}
A simple calculation also shows that $\pi_v(f)t_v (F)\pi_v(g) = t_v (fFg)$, for all $f ,g\in C_0 (\Gvertex)$ and $F \in C_c(\Gedge)$. Hence $(\pi_v, t_v)$ extends to an isometric covariant representation of the correspondence $X_{\G}$.

Let $\pi \equiv \oplus_{v \in \Gvertex} \pi_v $ and $t \equiv \oplus_{v \in \Gvertex} t_v$. We will show now that the $\ca$-algebra $\ca(\pi, t)$ generated by $\pi(C_0 (\Gvertex))$ and $t(X_{\G})$ is isomorphic to the Toeplitz-Cuntz-Pimsner $\ca$-algebra $\T_{\G}$. From this it will follow that the family of representations $\phi_v$, $v \in \Gvertex$, separates the points in $\T_{\G}^{+}$, since $\T_{\G}^{+} \subseteq \T_{\G}$.

 By the Gauge Invariant Uniqueness Theorem of Katsura \cite[Theorem 6.2]{KatsuraJFA}, it suffices to show that $\pi$ is faithful on $C_0 (\Gvertex)$, the $\ca$-algebra $\ca(\pi , t)$ admits a gauge action and finally that \[
 I\left( \pi , t \right) \equiv \{ f \in C_0 (\Gvertex)\mid  \pi(f) \in \psi_t(K(X_{\G}))\}=0.
 \]

 Since $\pi_v(f)\xi_v=f(v)\xi_v$, $v \in \Gvertex$, it is clear that $\pi$ is faithful. To define a gauge action on $\ca(\pi , t)$ let
 \[
 \alpha_{\lambda}: \ca(\pi , t) \longrightarrow \ca(\pi , t): A \longmapsto U_{\lambda}^* AU_{\lambda}, \quad \lambda \in \bbT,
 \]
 where $U_{\lambda}\xi_{p} = \lambda^{|p|}\xi_{p}$, $p \in \G_{v}^{\infty}$ and $|p|$ equals  the length of the path $p$. Finally, for any $G \in X_{\G}$, we have $t_v(G)^*\xi_v = 0$, for all $v \in \Gvertex$. However if $\pi(f)\in I\left( \pi , t \right)$, then $\pi_v(f)$ can be approximated by linear combinations of elements of the form
 \[
 t_v(F_1)\dots t_v(F_k)t_v(G_1)^*\dots t_v(G_k)^*, \quad k = 1, 2, \dots ..
 \]
Since $t_v (G)^*\xi_v = 0 $, for all $G \in X_{\G}$, we have
\[|f(v)| = \| f(v)\xi_v\| = \|\pi_v(f)\xi_v \| = 0,
\]
for all $v \in \Gvertex$. Hence $f=0$, as desired.

It remains to verify the validity of (\ref{overl}).
Let $v_1, v_2, \dots$ be an enumeration of $\Gxvertex$. Fix an $i \in \bbN$; for each $n \in \bbN$ use Urysohn's Lemma to obtain $f_n =f_{n,i}\in C_0(\Gxvertex)$ with
\[
f_n(v_i)=1, \quad f_n(v_j)=0, \mbox{ for } 1\leq j \leq n, \, j\neq i,
 \]
and $|f_n(v_j)|\leq 1$, for all other $j$. If $\hat{L}_{v_i}$ is any weak limit of $\{ \pi_v(f_n)\}_n$ then $\hat{L}_{v_i}\xi_p = \xi_p$, if $r(p)=v_i$, or otherwise $\hat{L}_{v_i}\xi_p =0$. Hence, $\hat{L}_{v_i}= L_{v_i}$, for all $i$. Furthermore,
\begin{equation} \label{equality}
\overline{\phi_v(\T_{\G}^+)}^{\textsc{ wot}} \ni t_v(F)L_{v_i}=\sum_{e \in s^{-1}(v_i )}\, F(e)L_e \in \L_{\G}.
\end{equation}
However $\sum_{i=1}^{\infty} \, L_{v_i}=I$ and so (\ref{equality}) implies that
\[
t_v(F) =  \sum_{i=1}^{\infty} \, t_v(F) L_{v_i} \in \L_{\G}, \quad F \in C_c(\Gxedge),
\]
i.e., $\overline{\phi_v(\T_{\G}^+)}^{\textsc{ wot}} \subseteq \L_{\G}$.
To obtain the reverse inclusion, let $e_1, e_2, \dots$ be an enumeration of $\Gxedge$. Fix an $i \in \bbN$; for each $n \in \bbN$ use Urysohn's Lemma to obtain $F_n \in C_c(\Gxedge)$ with
\[
F_n(e_i)=1, \quad F_n (e_j)=0, \mbox{ for } 1\leq j \leq n, \, j\neq i,
 \]
and $|F_n(e_j)|\leq 1$, for all other $j$. If $v_{j_i} = s(e_i)$, then again by (\ref{equality}) any weak limit of $\{ t_v(F_n)L_{v_{j_{i}}} \}_{n}$ will equal $L_{e_i}$ and the conclusion follows.
\end{proof}

\begin{remark}
The representations $\phi_v$ in the proof of Theorem \ref{nestrepn} were inspired by the familiar orbit representations (see \cite{DavKatMem, Pet}). However, it is frequent the case that these two classes of representations differ.

Indeed, let $(\Omega, \sigma)$ be a dynamical system on a locally compact Hausdorff space $\Omega$ and recall that Peters' semicrossed product $C_0(\Omega) \times_{\sigma} \bbZ^+$ is the tensor algebra corresponding to the topological graph $\G= (\G^{0}, \G^{1}, r , s)$ with $ \G^{0} =  \G^{1} =\Omega $, $s = \id$ and $r = \sigma$. Let $v \in \Omega$ and let $\H$ be a separable Hilbert space with orthonormal base $\{ \xi_n \}_{n=0}^{\infty}$. The \textit{orbit representation} $\lambda_v$ of $C_0(\Omega) \times_{\sigma} \bbZ^+$ is defined by
\begin{align*}
\lambda_v(f)\xi_n &= f(\sigma^{(n)}(v))\xi_n, \, \, f \in C_0(\Omega),\,  n=0,1,2,\dots ,  \\
 \lambda_v(V)\phantom{.}&=S,
 \end{align*}
where $V \in C_0(\Omega) \times_{\sigma} \bbZ^+$ is the universal isometry and $S$ the forward shift on $\{ \xi_n \}_{n=0}^{\infty}$.

It is easy to see that when $v \in \Omega$ is aperiodic (Definition \ref{varpoints}), then $\lambda_v\left( C_0(\Omega) \times_{\sigma} \bbZ^+\right)$ is $\sot$-dense the nest algebra of all lower triangular infinite matrices with respect to $\{ \xi_n \}_{n=1}^{\infty}$. Therefore,
$\lambda_v \left( C_0(\Omega) \times_{\sigma} \bbZ^+\right)'' = B(\H)$. On the other hand, the range of any of the $\phi_{v'}$, $v'\in \Omega$, constructed above is $\sot$-dense in a free semigroupoid algebra and so \cite[Corollary 4.5]{KrP} implies that $\phi_{v'}\left(C_0(\Omega) \times_{\sigma} \bbZ^+\right)''$ is a free semigroupoid algebra. Since $B(\H)$ is not a free semigroupoid algebra, no such $\phi_{v'}$, $v'\in \Omega$, can be unitarily equivalent to $\lambda_v$, when $v$ is aperiodic.
\end{remark}

In \cite{LP} it was observed that the representation theory of Davidson and the author \cite{DavKatProcLMS} has definitive applications on the theory of local maps for graph algebras. The following extends the results of \cite{LP} to a much broader class of operator algebras.

\begin{theorem} \label{topolgraphmult}
If $\G=  (\G^{0}, \G^{1}, r , s)$ is a topological graph, then any approximately local left multiplier on $\T_{\G}^+$ is actually a left multiplier, i.e., the algebra $LM(\T_{\G}^+)$ of left multipliers on $\T_{\G}^+$ is reflexive.
\end{theorem}

\begin{proof} Let $S$ be an approximately local multiplier on $\T_{\G}^+$ By the previous Theorem, there exists a separating family of representations
\[
\rho_i : \T_{\G}^+\rightarrow B(\H_i), \quad i \in \bbI
\]
on finite dimensional Hilbert space so that each $\rho_i(\T_{\G}^+)$ is a finite dimensional nest algebra. Furthermore, $\rho_i(\T_{\G}^+)$ is a right $\T_{\G}^+ / \ker \rho_i $-module, with the right action coming from $\rho_i$. Since $S$ preserves closed left ideals we obtain \[
S_i: \T_{\G}^+ / \ker \rho_i \longrightarrow \rho_i (\T_{\G}^+ ); A + \ker \rho_i \longmapsto \rho_i(S(A)), \quad i \in \bbI.
\]
It is easy to verify that $S_i$ is an an approximate left multiplier and so Theorem \ref{basic1} implies that $S_i$ is a actually a left multiplier. Hence $$S_i (AB + \ker \rho_i)= S_i(A+\ker \rho_i)\rho_i(B)$$ and so
\[
\rho_i\big( S(AB)-S(A)B \big) =0, \quad \mbox{for all } i \in \bbI.
\]
Since $\cap_i \ker \rho_i= \{0\}$, the conclusion follows.
\end{proof}

\begin{corollary} \label{multivmult}
Let $(\Omega , \sigma)$ be a multivariable dynamical system and let $\T^+(\Omega , \sigma)$ be its tensor algebra. Then any approximately local left multiplier on $\T^+(\Omega , \sigma)$ is a left multiplier.
\end{corollary}

There is another class of operator algebras associated to a multivariable dynamical system $(\Omega , \sigma)$, the so called semicrossed product $C_0(\Omega) \times_{\sigma} \bbF_n^+$, which was introduced and first studied by Davidson and the author in \cite{DavKatMem}. We do not know whether approximately local left multipliers are left multipliers in general but the following provides an answer in a particular situation.

\begin{proposition} \label{RFD}
Let $(\A , \sigma)$ be an automorphic multivariable $\ca$-dynamical system and assume that the crossed product $\ca$-algebra $\A \times_{\sigma} \bbF_n$ is residually finite dimensional \textup{(RFD)}. Then any approximately local left multiplier on $\A \times_{\sigma} \bbF_n^+$ is a left multiplier.
\end{proposition}

\begin{proof} Let $U_1, U_2, \dots , U_n$ be the free unitaries in $\A \times_{\sigma} \bbF_n$ implementing the covariance relations.  Let $\pi_{i}$, $i \in \bbI$, be a family of irreducible representations that separates the points in $\A \times_{\sigma} \bbF_n$ and let $\hat{\pi}_{i}$, $i \in \bbI$, be the restriction of that family on $\A \times_{\sigma} \bbF_n^+$. Since $\pi_i (\A \times_{\sigma} \bbF_n^+)$ is an operator algebra acting on a finite dimensional space, it is inverse closed and therefore it contains the adjoints of the unitary operators $\pi_i( U_j)$, $i=1, 2, \dots, n$. Hence $\pi_i (\A \times_{\sigma} \bbF_n^+)$ equals $\pi_i (\A \times_{\sigma} \bbF_n)$  and so $\hat{\pi}_{i}$, $i \in \bbI$, is a separating family of finite dimensional irreducible representations for $\A \times_{\sigma} \bbF_n^+$. As with Theorem~\ref{topolgraphmult}, this suffices to prove the result.
\end{proof}

Proposition \ref{RFD} applies in particular to the case where $\A = \bbC$ and $\sigma$ consisting of $n$-copies of the identity map. Indeed

\begin{corollary} \label{Choi}
Let $\ca(\bbF_n)$ be the full $\ca$-algebra of the free group on $n$-generators and let $A (\bbF_n^+)$ be the non-selfadjoint subalgebra of  $\ca(\bbF_n)$ generated by the generators. Then any approximately local multiplier on $A (\bbF_n^+)$ is a multiplier.
\end{corollary}

\begin{proof}
The result follows from a well-known result of Choi \cite[Theorem 1]{Choi} and Proposition \ref{RFD}.
\end{proof}

An interesting result of Samei \cite{Samei} asserts that if all approximately local one-sided multipliers from a Banach algebra $\A$ into its one-sided modules are one-sided multipliers, then all approximately local derivations on $\A$-bimodules are derivations. This approach was used by Hadwin and Li \cite{HadL} to prove that all local derivations on CSL algebras are actually derivations by proving the corresponding result for approximately one-sided multipliers.

It turns out that this approach is not applicable here. Neither Theorem~\ref{topolgraphmult} nor Corollary \ref{multivmult} do extend to approximate left multipliers on a right $\T_{\G}^+$-module, instead of $\T_{\G}^+$ itself. The following provides a broad class of counterexamples.

\begin{example}
Let $(\Omega , \sigma)$ be a dynamical system on a locally compact Hausdorff space and assume that $\sigma$ has a fixed point $x  \in \Omega$. Then there exists a finite dimensional $C_0(\Omega) \times_{\sigma} \bbZ^{+}$-module $\fA$ and a local left multiplier $S: C_0(\Omega) \times_{\sigma} \bbZ^{+} \rightarrow \fA$ which is not a left multiplier.

Indeed, let $\fA$ and $S_{\fA}$ be as in Example \ref{nonmult}. Since $x  \in \Omega$ is a fixed point for $\sigma$ the mapping
\[
\pi : C_0(\Omega) \times_{\sigma} \bbZ^{+} \longrightarrow \fA;\, \sum_{n=0}^{\infty}V^nf_n\longmapsto
\begin{pmatrix}
f_0(x) & f_1(x) \\
0 & f_0 (x)
\end{pmatrix}
\]
is a representation that makes $\fA$ a right $C_0(\Omega) \times_{\sigma} \bbZ^{+}$-module with the natural action coming from $\pi$. If
\[
S: C_0(\Omega) \times_{\sigma} \bbZ^{+} \longrightarrow \fA;\, A \longmapsto S_{\fA}(\pi(A)),
\]
then $S$ is a local left multiplier which is not a left multiplier (otherwise $S_{\fA}$ would be a multiplier as well).
\end{example}

%%%%%%%%%%%%%%%%%%%%%%%%%%%%%
\section{Local Derivations}
%%%%%%%%%%%%%%%%%%%%%%%%%%%%%%%%%%%%

In this section we investigate whether an analogue of Theorem \ref{topolgraphmult} is valid for approximately local derivation instead of multipliers. This is a harder problem and the following illustrates where the difficulty lies.

\begin{proposition} Let $\A$ be an operator algebra and assume that
\begin{itemize}
\item[(i)] the finite dimensional nest representations separate points in $\A$
    \item[(ii)] any derivation $\delta:\A \rightarrow \A$ leaves invariant any closed ideal.
    \end{itemize}
Then any approximately local derivation $\delta: \A \rightarrow \A$ is a derivation.
\end{proposition}

\begin{proof}
Let $\delta$ be an approximately local derivation on $\A$ and let $\{ \phi_i\}_{i \in \bbI} $ be a family of finite dimensional nest representations that separate points in $\A$. Since locally $\delta$ can be approximated by derivations, it leaves all closed ideals invariant. Hence we obtain well-defined approximately local derivations
\[
\delta_{i}: \A  /  \ker\phi_{i} \longrightarrow  \phi_{i}(\A); A + \ker  \phi_{i} \mapsto  \phi_{i} (\delta(A)), \quad i \in \bbI,
\]
with the right action on $\phi_{i} (\A)$ coming from $\phi_i$. Since $\A  /  \ker\phi_{i} $ is generated by its idempotents (it is a finite dimensional nest algebra), Theorem \ref{basic2} implies that the $\delta_{i}$ are actually derivations. Hence
\[
\phi_{i}\big( \delta(AB)-\delta(A)B - A\delta(B)\big) =0,
\]
for all $ i \in \bbI$. Since $\cap_{i\in \bbI}\ker\phi_{i}= \{0\}$, the conclusion follows.
\end{proof}

If we knew that any derivation on the tensor algebra of a topological graph leaves closed ideals invariant, e.g., it is inner, then the result above combined with Theorem \ref{nestrepn} would imply that local derivations are derivations. With the exception of \cite{Duncan}, it seems that nothing is known in that direction. We therefore adopt a different approach: using the concept of an acyclic graph, we build a more flexible representation theory than that of Theorem \ref{nestrepn}. The drawback is that it is not applicable in all situations.

We begin by recalling some further results regarding the structure of the tensor algebra of a topological graph $\G=  (\G^{0}, \G^{1}, r , s)$. Let
\[
\G^n \equiv \{ e_ne_{n-1}\dots e_1 \mid e_i \in \G^1, s(e_{i+1})=r(e_i), \, 1\leq i \leq n \}
\]
be the space of paths of length $n$ equipped with the topology it inherits a a subset of $\G^1 \times  \G^1 \times \dots \times \G^1 $. Equip $\G^n $ with the obvious domain and range maps $s: \G^n \rightarrow \G^0 $ and $r: \G^n \rightarrow \G^0 $, i.e., $s(e_ne_{n-1}\dots e_1) = s(e_1)$ and  $r(e_ne_{n-1}\dots e_1) = r(e_n)$. Then \cite[Lemma 1.25]{Katsura} shows that the quadruple $(\G^{n}, \G^{0}, r , s)$ becomes a topological graph which we simply denote as $\G^n$. Furthermore, \cite[Proposition 1.27]{Katsura} shows that $X_{\G^n} \simeq X_{\G}\otimes  X_{\G} \otimes \dots \otimes  X_{\G}$, via the identification
\begin{equation} \label{identification}
\big( F_n\otimes F_{n-1} \otimes \dots \otimes F_1\big) (e_ne_{n-1}\dots e_1 ) = F_{n}(e_n)F_{n-1}(e_{n-1})\dots F_1(e_1),
\end{equation}
where $F_i \in X_{\G}$ , $i=1,2,  \dots , n$, $e_ne_{n-1}\dots e_1 \in \G^n$ and $\otimes$ denotes the internal tensor product of $\ca$-correspondences.

Now fix a $v \in \G^0$ so that the graph $\G_v$ is acyclic, i.e., no paths of length higher than zero start and end at the same vertex, and recall the representation $\phi_v = \pi_v \times t_v$ of $\T^+_{\G}$, as it appears in the proof of Theorem \ref{nestrepn}. Using $\phi_v$, we construct now a family of finite dimensional representations for $\T_{\G}^+$ whose kernels satisfy a useful property.

For each $n= 0, 1, 2, \dots$, let
\[
\F_{v,n} = \left\{ p \in \G^{\infty}_v \, \Big| \,  s(p)\in  \bigcup_{i=0}^{n} \G_{v, i}^0 \mbox{ and } r(p) \in  \bigcup_{i=0}^{n} \G_{v, i}^0  \right\}
\]
and
\[
\H_{v, n} =\bigvee  \left \{ \xi_p \mid p \in \F_{v,n}  \right\} \subseteq \H_{\G_v}.
\]
It is easy to verify that $\H_{v, n}$ is co-invariant for $\phi_v(\T^+_{\G})$ and so the compression of $\phi_v$ on $\H_{v, n}$ defines a representation of $\T^+_{\G}$ denoted as $\phi_{v,n}$. Since $\bigvee_n \H_{v,n}=\H_{\G_v}$, the family $\{ \phi_{v, n}\}_{n\in \bbN}$ separates points in $\T_{\G}^+/\ker\phi_v$.

In order to ensure that the representations appearing in the next lemma are finite dimensional we need to assume that the topological graph $\G=  (\G^{0}, \G^{1}, r , s)$ is \textit{source finite}, i.e., $s^{-1}(\{v\})$ is a finite set for any $v \in \G^0$. This assumption is automatically satisfied when $\G^1$ is compact and in all applications considered in this paper.

\begin{lemma} \label{idemotrep}
 If $\phi_{v,n}$ is as above, with $\G_v$ acyclic, then $\phi_{v,n}(\T^+_{\G})$ is a finite dimensional operator algebra which is generated by its idempotents.
\end{lemma}

\begin{proof}
Since $\G_v$ is acyclic, $\G_{v,i}^0 \cap \G_{v,j}^{\,0} = \emptyset$, for $i \neq j$. This implies that the cardinality of $\F_{v, n}$
is finite. Hence $\H_{v, n}$ is finite dimensional and the algebra $\phi_{v,n}(\T^+_{\G})$ is generated as a linear space by the (finitely many) compressions $\hat{L}_{p}\equiv P_{\H_{v, n}}L_{p}|_{\H_{v, n}}$, $p \in \F_{v, n}$.

If $p \in \F_{v, n}
\cap \G_v^0$, then $\hat{L}_{p}$ is an idempotent (actually an orthogonal projection). Otherwise, $\hat{L}_{p} = \hat{L}_{r(p)} \hat{L}_{p} \hat{L}_{s(p)}$ and since $\G_v$ is acyclic we have $$\hat{L}_{r(p)} \hat{L}_{s(p)} = 0. $$ Hence the identity
\[
2 \hat{L}_{p} =  \big(\hat{L}_{r(p)} +\hat{L}_{p} \big) -  \big(\hat{L}_{r(p)} -\hat{L}_{p} \big)
\]
shows that $\hat{L}_{p}$ belongs to the span of two idempotents and the conclusion \break follows.
\end{proof}

 If $\fG= (\fG^0, \fG^1, r, s)$ is a topological graph and  $S \subseteq \fG^1 $, then $N(S)$ denotes the collection of continuous functions $F \in X_{\fG}$ with $F_{|S}=0$, i.e., vanishing at $S$.

 \begin{lemma} \label{topology}
 Let $\fG= (\fG^0, \fG^1, r, s)$ be a topological graph.
 \begin{itemize}
 \item[(i)] If $S_1, S_2 \subseteq \fG^1$ closed, then
 \[
 N(S_1 \cap S_2) = \overline{N(S_1) + N(S_2)}.
 \]
 \item[(ii)] If $S_1 \subseteq \fG^0$, $S_2 \subseteq \fG^1$ closed, then
 \[
 N(r^{-1}(S_1) \cup S_2)=\overline{\phantom{,}\big[ \{(f\circ r) F\mid f_{|S_1} = 0, F_{|S_2} = 0\}\big]\phantom{,}}
 \]
 \item[(iii)] If $S_1 \subseteq \fG^0$, $S_2 \subseteq \fG^1$ closed, then
 \[
 N(s^{-1}(S_1) \cup S_2)=\overline{\phantom{,}\big[ \{(f\circ s) F\mid f_{|S_1} = 0, F_{|S_2} = 0\}\big]\phantom{,}}
 \]
 \end{itemize}
 \end{lemma}

 \begin{proof}
 Let $K\subseteq \fG^1$ compact. We claim that for any sequence $\{F_n\}_n$ with $\supp F_n \subseteq K$, $n\in \bbN,$ convergence in $X_{\fG}$ is equivalent to convergence in the usual supremum norm $\|\phantom{x}\|_{\infty}$ of $C_0 (\fG^1)$.

 Indeed, this will follow if we show that there exists $n \in \bbN$ so that \break $\| F\| \leq n \| F \|_{\infty}$, for any $F$ with $\supp F \subseteq K$. Let $U_1, U_2, \dots , U_{n}$ be an open cover of $K$ so that $s_{|U_i}$ is injective for all $i=1, 2, \dots , n$. Notice that for any function $G$, with $\supp G\subseteq U_i$ for some $i$, we have $\|G\| = \|G\|_{\infty}$. Hence, if $\{ H_i \}_{i=1}^n$ is a partition of unity for $K$ subordinated by the cover $U_1, U_2, \dots , U_{n}$, then,
 \[
 \|F\| \leq\sum_{i=1}^n \|H_i F\|= \sum_{i=1}^n \|H_i F\|_{\infty} \leq n\|F\|_{\infty}
 \]
 and the conclusion follows.
 %%%%%%%%%%%%%%%%%%%%%%%%%%
 %Let $U$ a compact neighborhood of $K$. For each $v \in s(K)$ there exist an open set $V_v \subseteq \fG^0$ and open sets $U_1, U_2, \dots , U_{n_x}$ containing $s^{-1}(\{x\})$ and $s_{|U_i}$ is injective onto $V_x$. By shrinking $V_x$ if necessary, we may assume that
 %\[s^{-1}\big(V_x\big)\cap K \subseteq \cup_{i=1}^{n_x} U_{i}.
 %\]
 %(Otherwise, there would be a net $\{ v_a \}$ in $V_v$ converging to $v$ with $e_a \in s^{-1}(v_a) \cap K \backslash  \cup_{i=1}^{n_x} U_{i}$. Any limit point $e$ of $\{ e_a \}$ then satisfies $s(e)= v$ and $e\notin \cup_{i=1}^{n_x} U_{i}$, a  contradiction.) Cover $s(K)$ with finitely many $V_{x_1}, V_{x_2}, \dots V_{x_{m_{0}}}$. Then $n_0= \max \{n_{x_1}, n_{x_2}, \dots , n_{x_{0}}\}$.
%%%%%%%%%%%%%%%%%%%%%%%%%%%%%%%%%%%%%%%%%%%

 (i) Let $F \in N(S_1 \cap S_2)$. By \cite[Lemma 1.6]{Katsura} and its proof, $F$ can be approximated by functions of the form $HF$, where $H \in C_c (\fG^1)$. Clearly such functions belong to $N(S_1 \cap S_2)$ and so without loss of generality we may assume that $\supp F = K$ compact. The proof now follows familiar lines as we only need to approximate $F$ by elements of $N(S_1) + N(S_2)$, in the usual supremum norm, while staying inside $K$.

 Let $\epsilon >0$ and let $K_{\epsilon} = \{ e \in K\mid |F(e)\geq \epsilon\}$. Since $K_{\epsilon}$ is disjoint from $S_1\cap S_2$, we can cover $K_{\epsilon}$ with finitely many $U_1, U_2, \dots U_m$ so that each of the $U_i$ is disjoint from one of the $S_1, S_2$. Let $H_i$, $i=1, 2, \dots , m $ a partition of unity for $K_{\epsilon}$ which is subordinate to $U_1, U_2, \dots U_m$. Then $\sum_{i=1}^{m} \, H_iF \in N(S_1) + N(S_2)$ and is $\epsilon$-close to $F$.

 (ii) Let $F \in N(r^{-1}(S_1) \cup S_2)$. Once again, by \cite[Lemma 1.6]{Katsura} there is no loss of generality assuming that $F$ has compact support. If $\{ f_i \}_{i \in \bbI}$, is an increasing approximate unit for $C_0 (\fG^0 \backslash S_1) \subseteq C_0(\fG^0)$, then Dini's Theorem implies that the family $\{(f_i\circ r)F\}_{i \in \bbI}$ approximates $F$ (in both norms) and proves the Lemma.

 (iii) The proof is similar to that of (ii).
 \end{proof}

Lemma \ref{topology} is now being used in the following with $\G^i$, $i \in \bbN$, in the place of $\fG$.

 \begin{lemma}  \label{kernelsquare}
 If $\phi_{v, n}$ is as above, with $\G_v$ acyclic, then $$\ker \phi_{v, n} = \overline{\, [ (\ker \phi_{v, n})^{2}] }.$$
 \end{lemma}

\begin{proof}
From the Fourier series expansion of (\ref{Cesaro}) and the results of Katsura discussed in the beginning of this section, it follows that each elements of $\T^+_{\G}$ admits a Fourier series expansion of the form
\begin{equation} \label{Fourier}
\pi_{\infty}(f) + \sum_{i=1}^{\infty} \,  t_{\infty}^i(F_i),
\end{equation}
with $f \in C_0(\G^0)$ and $F_i \in X_{\G^i }$, $i=1,2, \dots$. It is easy to see now that $\ker \phi_{v, n}$ consists precisely of all elements of $\T^+_{\G}$, whose Fourier series (\ref{Fourier}) satisfies
\[
f \in N\big(\{p \in \F_{v,n} \mid |p|=0\}\big)
\]
and
\[
F_i \in N\big(\{p \in \F_{v,n} \mid |p|=i\}\big), \quad i=1,2,  \dots
\]
Note that in the case where $i>n$, the set $\{p \in \F_{v,n} \mid |p|=i\}$ is empty and so the above simply says that $F_i$ ranges over $X_{\G^i_v }$.

To prove the Lemma, fix an $i$ and notice that any finite sum of the form
\begin{multline} \label{ker2}
\sum_m\Big( \pi_{\infty}(f_m) t^i_{\infty}(F_m) + t^i_{\infty}(F'_m)\pi_{\infty}(f_{m}')\Big) \\
= t^i_{\infty}\Big( \sum_m \, (f_m\circ r) F_m + \sum_m \, (f_m' \circ s) F'_m \Big)
\end{multline}
belongs to $[(\ker \phi_{v, n})^2]$, where $f_m, f_m' \in N(\cup_{i=0}^{n} \G_{v, i}^0)$ and $F_m , F_m' \in N(\{p \in \F_{v,n} \mid |p|=i\})$, for all $m$. However, Lemma \ref{topology} shows that sums of the form \break $\sum_m (f_m\circ r) F_m $ are dense in
\begin{equation} \label{id1}
N\Big( r^{-1}\Big(\bigcup_{j=0}^{n} \G_{v, j}^0\Big) \bigcup \left\{p \in \F_{v,n} \Big| |p|=i \right\} \Big),
\end{equation}
and similarly, sums of the form $\sum_m ( f_m' \circ s) F'_m ( f_m' \circ s)$ are dense in
\begin{equation} \label{id2}
N\Big(  \left\{p \in \F_{v,n} \Big| |p|=i \right\}\Big) \bigcup s^{-1}\Big(\bigcup_{j=0}^{n} \G_{v, j}^0\Big) \Big),
\end{equation}
where $r, s$ denotes the range and source functions on paths of length $i$.
Since $\G_v$ is acyclic
\[
r^{-1}\Big(\bigcup_{j=0}^{n} \G_{v, j}^0\Big) \bigcap s^{-1}\Big(\bigcup_{j=0}^{n} \G_{v, j}^0\Big) = \left\{p \in \F_{v,n} \Big| |p|=i \right\}
\]
and so Lemma \ref{topology}(i) shows that the closed linear space generated by (\ref{id1}) and (\ref{id2}) equals $\{p \in \F_{v,n} \big| |p|=i \}$. This shows that the sums as in (\ref{ker2}) can approximate the $i$-th Fourier coefficient of any element in $\ker \phi_{v, n}.$
\end{proof}

Another possibility that we will be considering here is that the graph $\G_v$ may be \textit{transitive}, i.e., given any two vertices $v', v'' \in \Gxvertex$ there exists a path starting at $v'$ and ending at $v''$. In that case, Davidson and the author have shown \cite[Theorem 4.4]{DavKatProcLMS} that $\L_{\G_v}$ admits a separating family of finite dimensional irreducible representations. Composing these representations with the representation $\phi_v$ of Theorem \ref{nestrepn}, we obtain a separating family $\phi_{v , n}$, $n \in \bbN$,  of finite dimensional irreducible representations of $\T^+_{\G} / \ker \phi_{v}$

\begin{theorem} \label{mainder}
 Let $\G=  (\G^{0}, \G^{1}, r , s)$ be a source finite topological graph and let $\{\G_v \}_{v \in \G^0}$ be the family of discrete graphs associated with $\G$. Assume that the set of all points $ v \in \G^0$ for which $\G_v$ is either acyclic or transitive, is dense in $\G^0$. Then any approximately local derivation on $\T^+_{\G}$ is a derivation.
\end{theorem}

\begin{proof} Let $\delta$ be an approximately local derivation. Consider the representations $\phi_{v, n}$, discussed above, where $v$ ranges over all points for which $\G_v$ is either acyclic or transitive and $n \in \bbN$. Because of the density assumption this family separates points.

\vspace{.14 in}
\noindent
\textbf{Claim:} \textit{ For any approximately local derivation $\hat{\delta}: \T^+_{\G} \rightarrow T^+_{\G}  $, we have  $$\hat{\delta}(\ker \phi_{v, n} ) \subseteq \ker \phi_{v, n}. $$}
%\vspace{0.05in}
\noindent Indeed let $ A \in \ker \rho_{x, n} $ and let $\delta_A$ be a derivation of  $T^+_{\G}$ so that $\hat{\delta}(A) = \delta_A(A)$. We distinguish two cases.

If $\G_v$ is transitive, then we have noted in the discussion proceeding  this theorem that $\phi_{v, n}$ is a finite dimensional irreducible representation and so $\ker \phi_{v, n}$ is a primitive ideal. Since derivations of Banach algebras leave primitive ideals invariant \cite[Proposition 6.4.16]{Pal}, $\hat{\delta}(A) = \delta_A(A) \in \ker \phi_{v, n}$.

On the other hand, if $\G_v$ is acyclic then Lemma~\ref{kernelsquare} shows that \break $(\ker \phi_{v, n})^2 = \ker \phi_{v, n}$. Hence any derivation of $\T^+_{\G} $ leaves
$\ker \rho_{x, n} $ invariant. In particular, this applies to $\delta_A$ and so $\hat{\delta}(A) = \delta_A(A) \in \ker \phi_{v, n}$ in this case as well.

\vspace{.14in}

The rest of the proof follows now familiar lines. By the Claim above, $\delta$ preserves $\ker \phi_{v, n}$ and so we obtain a map
\[
\delta_{x, n}: \T^+_{\G}  /  \ker\phi_{v, n} \longrightarrow  \phi_{v, n}(\T^+_{\G} ); A + \ker  \phi_{v, n} \mapsto  \phi_{v, n} (\delta(A)).
\]
It is easy to see, again from the Claim above, that $\delta_{x, n}$ is an an approximate local derivation. Furthermore, by Lemma \ref{idemotrep}, $\T^+_{\G}  /  \ker\phi_{v, n} $ is generated by its idempotents and so Theorem \ref{basic2} implies that $\delta_{x, n}$ is a actually a derivation. Hence
\[
\phi_{v, n}\big( \delta(AB)-\delta(A)B - A\delta(B)\big) =0.
\]
Since $\cap_{v,n}\ker\phi_{v, n}= \{0\}$, the conclusion follows.
\end{proof}

The first significant application of the above result is to Peters' semicrossed products, where nothing was previously known regarding local derivations.

\begin{definition} \label{varpoints} Let $(\Omega , \sigma)$ be a dynamical system on a locally compact Hausdorff space. A point $ x \in \Omega$ is said to be \textit{periodic} if there exists $i \in \bbN$ so that $\sigma^{(i)}(x) = x$. The point $ x \in \Omega$ is said to be \textit{eventually periodic} if $x$ is not periodic but one of its iterations  $\sigma^{(i)}(x)$, $i \in \bbN$ is periodic. Otherwise $x$ is said to be \textit{aperiodic}.
\end{definition}

Let $(\Omega , \sigma)$ be a dynamical system on a locally compact Hausdorff space. The semicrossed product $C(\Omega) \times_{\sigma} \bbZ^{+}$ is the tensor algebra corresponding to the topological graph $\G= (\G^{0}, \G^{1}, r , s)$ with $ \G^{0} =  \G^{1} =\Omega $, $s = \id$ and $r = \sigma$. If $v  \in  \Omega$  then
\[
\G_{v}^{0} = \G_{v}^{1} = \O_v \equiv \{ v , \sigma(v), \sigma^{(2)}(v),  \dots \},
\]
ie., the orbit of $v$. If $v$ is periodic with period $n$, then the corresponding graph $\G_v$ is the $n$-cycle graph $\C_n$ that has been studied quite extensively in the graph algebra theory; this is a transitive graph. On the other hand, if $v$ is aperiodic then $\G_v$ is an infinite countable graph with exactly one edge starting from each $\sigma^{(n)}(v)$ (and ending at $\sigma^{(n+1)}(v)$), $n =, 0, 1, 2, \dots$. This graph is acyclic.

There also a third possibility that the point $v \in \Omega$ is eventually periodic for $\sigma$. In that case the graph $\G_v$ is a  combination of the above graphs. If $k \in \bbN$ is the least positive integer so that $\sigma^{(k)}(v)$ is periodic, then at $\sigma^{(k)}(v)$ the graph $\G_v$ supports an $n$-cycle graph ($n$ is the period of $\sigma^{(k)}(v)$) and receives a tail starting at $v$. We do not know how to deal with that case as it is not covered by Theorem \ref{mainder} and so in the next result we assume that such points are topologically insignificant.

\begin{corollary} \label{semicrossder}
Let $(\Omega , \sigma)$ be a dynamical system on a locally compact Hausdorff space $\Omega$. Assume that the eventually periodic points of $\sigma$ have empty interior, e.g, $\sigma$ is a homeomorphism. Then any approximately local derivation on $C(\Omega) \times_{\sigma} \bbZ^{+}$ is a derivation. Hence $\Z^1(C(\Omega) \times_{\sigma} \bbZ^{+})$ is reflexive.
\end{corollary}

Let $\Omega$ be a compact Hausdorff space and let $\sigma: \Omega \rightarrow \Omega$ be a covering map. Consider the topological graph $\G_{\text{EV}}(\Omega, \sigma)=  (\G^{0}, \G^{1}, r , s)$, where $\G^0 = \G^1 = \Omega$, $r = \id$ and $s = \sigma$. Assume further that the system $(\Omega, \sigma)$ is \textit{topologically free}, i.e., the sets $\Omega_{m ,n } \equiv \{ v \in \Omega \mid \sigma^{(m)}(v) = \sigma^{(n)}(v)\}$ have empty interior for all $m \neq n$. It is known \cite[Theorem 6.1]{BRV} that the Cuntz-Pimsner $\ca$-algebra associated with $\G_{\text{EV}}(\Omega, \sigma)$ is isomorphic to the Exel crossed product $\ca$-algebra $C( \Omega )\rtimes _{\sigma, L} \bbN$, where $L$ is the transfer operator that averages over inverse images of points. The condition of topological freeness was isolated by Exel and Vershik in their original work \cite{EV} and was further studied in \cite{BR, BRV, CS}. It is consequence of these works that the tensor algebra $\T^+_{\text{EV}}(\Omega, \sigma)$ for $\G_{\text{EV}}(\Omega, \sigma)$ admits the following elegant description

\begin{example} A faithful representation for a $\T^+_{\text{EV}}(\Omega, \sigma)$, when $(\Omega, \sigma)$ is topologically free.

Let $\H$ be a Hilbert space with orthonormal base $\{e_v\}_{v \in \Omega}$. For $f \in C(\Omega)$ let $M_f$ be the multiplication operator $M_f (e_v) = f(v)e_v$, $v \in \Omega$, and
\[
S(e_v) = \big(L(1_{\Omega})(v)\big)^{-1/2}\sum_{w \in \sigma^{-1}(\{v\})} \, e_w
\]
In \cite[Theorem 6]{CS} it is shown that the $\ca$-algebra generated by $M_f$, $f \in C(\Omega)$, and $S$ is canonically isomorphic to $C( \Omega )\rtimes _{\sigma, L} \bbN$, i.e., with generators going to generators. Furthermore, it is a consequence of \cite[Proposition 3.1]{BRV} and  \cite[Theorem 6.1]{BRV} that $C( \Omega )\rtimes _{\sigma, L} \bbN$ is once again canonically isomorphic to $\O_{\text{EV}}(\Omega, \sigma)$. This implies that $\T^+_{\text{EV}}(\Omega, \sigma)$ is isomorphic to the non-selfadjoint algebra generated by $M_f$, $f \in C(\Omega)$, and $S$.
\end{example}

\begin{corollary} Let $\Omega$ be a compact Hausdorff space and let $\sigma: \Omega \rightarrow \Omega$ be a topologically free covering map. If $\G_{\text{EV}}(\Omega, \sigma)=  (\G^{0}, \G^{1}, r , s)$, where $\G^0 = \G^1 = \Omega$, $r = \id$ and $s = \sigma$, then any approximately local derivation on $\T^+_{\text{EV}}(\Omega, \sigma)$ is a derivation.
\end{corollary}

\begin{proof}
 It is not difficult to see that for a point $v \in \Omega$, the corresponding discrete graph $\G_v$ is acyclic iff $v \notin \cup_{m \neq n} \Omega_{m, n}$. By the Baire Category Theorem, $\cup_{m \neq n} \Omega_{m, n}$ has empty interior and the conclusion follows from Theorem \ref{mainder}
 \end{proof}

Another application of Theorem \ref{mainder} takes place for certain tensor algebras that materialize as subalgebras of Exel's crossed products by partial automorphisms. Let $\Omega$ be a locally compact space and $\sigma$ a partial homeomorphism, i.e.,  $\sigma: U\rightarrow V$ is a homeomorphism between two open subset $U,V$ of $\Omega$. Let $\theta: C_0(U) \rightarrow C_0 (V)$ be the $*$-automorphism induced by $\sigma$. The triple $(\theta , C_0(U) , C_0(V))$ is called a partial automorphism of $C_0(\Omega)$. In \cite[Definition 3.7]{Exel} Exel associates a $\ca$-algebra $C_0(\Omega) \times_{\theta} \bbZ$ with the partial automorphism $(\theta , C_0(U) , C_0(V))$. This algebra is defined through a universal property but Katsura has shown that it corresponds to the Cuntz-Pimsner algebra corresponding to the graph $\G=  (\G^{0}, \G^{1}, r , s)$, where $\G_0 = \Omega$, $\G_1= U$, $r = \sigma$ and $s$ is the natural embedding.

\begin{corollary} \label{semicrossderExel}
Let $C_0(\Omega) \times_{\theta} \bbZ$ be Exel's crossed product by the partial automorphism $(\theta , C_0(U) , C_0(V))$ of $\Omega$ and let $C_0(\Omega) \times_{\theta} \bbZ^+$ be the associated tensor algebra. Then any approximately local derivation on $C_0(\Omega) \times_{\theta} \bbZ^+$ is a derivation.
\end{corollary}

We now focus on tensor algebras for multivariable dynamical systems. If  $(\Omega, \sigma)$ is such a system, with $\sigma = (\sigma_1, \sigma_2, \dots , \sigma_n)$, and $u = i_k i_{k-1} \dots i_1 \in \bbF_n^+$ then we write $\sigma_u \equiv \sigma_{i_k}\circ \sigma_{i_{k-1}}\circ \dots \circ \sigma_{i_1}$.

\begin{corollary}
Let $(\Omega, \sigma)$ be a multivariable system on a locally compact space $\Omega$ and assume that the set
\[
\Omega_0 = \Big \{ v \in \Omega \mid \sigma_u (v)=\sigma_w(v), \mbox{ for some } u,w \in \bbF_n^+ \mbox{ with }|u|\neq |w| \Big \}
\]
has empty interior. Then any approximately local derivation on $\T^+(\Omega, \sigma)$ is a derivation.
\end{corollary}
\begin{proof} It is easy to verify that $\G_v$ is acyclic provided that $v \in \Omega \backslash \Omega_0$.
\end{proof}

The assumptions of the above Corollary hold in particular when $\Omega = \bbT$, and $\sigma_k(z) = e^{2\pi i \theta_k}z$, $k=1,2$, with $\theta_1 , \theta_2$ irrational satisfying $\theta_1/ \theta_2 \notin \bbQ$.

\begin{corollary} \label{group}
Let $\fG$ be a finite group acting on a locally compact Hausdorff space $\Omega$ via a family of automorphisms $\sigma_{\fG} \equiv \{\sigma_g \}_{g \in \fG}$. Then any approximately local derivation on $\T^+(\Omega, \sigma_{\fG})$ is a derivation.
\end{corollary}

\begin{proof} We claim that for any $v \in \Omega$, the graph $ \G_v$ of Theorem \ref{nestrepn} is transitive, i.e., given any two vertices $v', v'' \in \Gxvertex$ there exists a path starting at $v'$ and ending at $v''$. Indeed by construction for any $v' \in \Gxvertex$ there is a path of the form
\[
(g_i ,\sigma_{ g_i^{-1}}(v')), \dots (g_2, \sigma_{g_1}(v))(g_1, v), \quad g_k \in \fG, k=1,2, \dots ,i
\]
starting at $v$ and ending at $v'$. Since $\fG$ is a group, the opposite path
\[
(g_1^{-1}, g_1(v))\dots(g_{i-1}^{-1}, g_i^{-1}(v'))(g_i^{-1}, v')
\]
exists thus establishing the existence of paths between $v$ and any $v'$ in both directions. Thus $\G_v$ is transitive.
The result follows now from \break Theorem~\ref{mainder}.
\end{proof}

We observed in Proposition \ref{RFD} that the semicrossed product $\A \times_{\sigma} \bbF_n^+$ for a multivariable $\ca$-dynamical system $(\A, \sigma)$ admits a separating family $\hat{\pi}_{i}$, $i \in \bbI$ of finite dimensional irreducible representations. Since the kernels of these representation are primitive ideals of $\A \times_{\sigma} \bbF_n^+$, any approximately local derivation leaves them invariant. Therefore by following the same steps as in the proof of Theorem \ref{semicrossder} we obtain

\begin{proposition} \label{RFDder}
Let $(\A , \sigma)$ be an automorphic multivariable \break $\ca$-dynamical system and assume that the crossed product $\ca$-algebra $\A \times_{\sigma} \bbF_n$ is residually finite dimensional \textup{(RFD)}. Then any approximately local derivation on $\A \times_{\sigma} \bbF_n^+$ is a derivation.
\end{proposition}

In particular, Proposition \ref{RFDder} applies to the universal algebra $A (\bbF_n^+) \subseteq \ca(\bbF_n)$ generated by $n$ contractions (see the proof of Proposition \ref{Choi}). Here is another application coming from \cite{DavKatMem}.

\begin{corollary}
Let $\sigma = (\sigma_1, \sigma_2)$ be a multivariable system on the two-point space $\Omega = \{0,1\}$, with $\sigma_1=\id$ and $\sigma_2(i) =i+1 \mod 2$, $i=0,1$, Then any approximately local derivation on $C(\Omega) \times_{\sigma} \bbF_2^+$ is a derivation.
\end{corollary}

\begin{proof} In \cite[Example 3.24]{DavKatMem}, we observed that $C(\Omega) \times_{\sigma} \bbF_2 \simeq \Mat_2 (\ca(\bbF_3))$, which is RFD, and so the conclusion follows from Proposition \ref{RFDder}.
\end{proof}

\section{concluding remarks and open problems}

Theorem \ref{mainder} does not extend to local derivations with range on a bimodule rather than the tensor algebra itself. The counterexample is a actually on a semicrossed product bimodule.

\begin{example} \label{countsemi}
Let $(\Omega , \sigma)$ be a dynamical system on a locally compact Hausdorff space and assume that $\sigma$ has a fixed point $x  \in \Omega$. Then there exists a finite dimensional $C_0(\Omega) \times_{\sigma} \bbZ^{+}$-module $\fB$ and a local derivation $\delta: C_0(\Omega) \times_{\sigma} \bbZ^{+} \rightarrow \fB$ which is not a derivation.

Indeed, let $\fB$ and $\delta_{\fB}$ be as in Example \ref{nonder}. Since $x  \in \Omega$ is a fixed point for $\sigma$ the mapping
\[
\pi : C_0(\Omega) \times_{\sigma} \bbZ^{+} \longrightarrow \fB;\, \sum_{n=0}^{\infty}V^nf_n\longmapsto
\begin{pmatrix}
f_0(x) & f_1(x) &f_2 (x)\\
0 & f_0 (x) & f_1(x) \\
0 & 0 & f_0(x)
\end{pmatrix}
\]
is a representation that makes $\fB$ a right $C_0(\Omega) \times_{\sigma} \bbZ^{+}$-module with the natural action coming from $\pi$. If
\[
\delta: C_0(\Omega) \times_{\sigma} \bbZ^{+} \longrightarrow \fB;\, A \longmapsto \delta_{\fB}(\pi(A)),
\]
then $\delta$ is a local derivation which is not a derivation (otherwise $\delta_{\fB}$ would be a derivation as well).
\end{example}

A modification of the above example also works for the non-commutative disc algebra $\A_n$, i.e., the tensor algebra of the graph with one vertex and $n$-loops, thus showing the failure of Theorem \ref{mainder} for local derivations on $\A_n$-modules as well. It turns out that using Example \ref{Crist}, we can construct yet another $\A_n$-module to demonstrate that failure.

\begin{example} Let $\A_n$, $n \geq 2$, be the non-commutative disc algebra with generators $A_1, A_2, \dots , A_n$ and let $\fC$, $\delta_{\fC}$ be as in Example \ref{Crist}. From the dilation theory for row contractions, it follows that there exists a representation $\pi:\A_n \rightarrow \Mat_3 (\bbC)$ so that
\[
\pi(A_1) = E_{1 2}, \pi(A_2)= E_{2 3} \mbox{ and } \pi(A_i) = 0, \mbox{ for all other } i.
\]
 If
\[
\delta: \A_n \longrightarrow \fC;\, A \longmapsto \delta_{\fC}(\pi(A)),
\]
then $\delta$ is a local derivation from $\A_n$ onto the $\A_n$-module $\fC = \pi(\A_n)$ which is not a derivation.
\end{example}

The above example shows that $\A_2$ admits local derivations on modules of dimension $3$ and $4$, which are not derivations. Does this persist on higher dimensions? One is tempted to guess that the answer is affirmative but we do not have a systematic way of producing such examples. What about semicrossed product modules? Can one construct such counterexamples  as Example \ref{countsemi} with modules of dimension higher than 3? Also we wonder whether any counterexamples do exist in the case where $\sigma$ is free.

It goes without saying that the following two open problems are tantalizing.

\begin{problem} Prove or disprove that a local left multiplier on the tensor algebra of a $\ca$-correspondence is actually a left multiplier.
\end{problem}

\begin{problem} Prove or disprove that an approximately local derivation on an arbitrary semicrossed product is actually a derivation.
\end{problem}

\vspace{0.1in}

{\noindent}{\it Acknowledgement.} A good part of this paper materialized while the author was visiting the Chinese Academy of Sciences in May 2014. The author would like to thank his hosts Professors Liming Ge and Wei Yuan for the stimulating conversations and their hospitality during his stay.

%%%%%%%%%%%%%%%%%%%%%%%%%%%%%%%%

\end{document}